\documentclass[12pt,a4paper]{article}
\usepackage[latin1]{inputenc}
\usepackage{amsmath, amssymb, enumerate, amsthm}

\usepackage{amsfonts}
\theoremstyle{definition}
\newtheorem{definition}{Definition}
\usepackage{graphicx}
\usepackage{float}
\usepackage{authblk}
\usepackage{subcaption} % loads the caption package

\newtheorem{theorem}{Theorem}
\usepackage[T1]{fontenc}
\usepackage{amsfonts}
\usepackage{graphicx}
\usepackage{booktabs}
\usepackage{multirow}
\usepackage{multicol}
\usepackage{lipsum}
\usepackage{mwe}
\usepackage{siunitx}
\usepackage{longtable} % To display tables on several pages
\usepackage{footnote}
\makesavenoteenv{tabular}
\makesavenoteenv{table}
\usepackage[font=small,labelfont=bf]{caption} % Required for specifying captions to tables and figures
\newcommand{\norm}[1]{\left\lVert#1\right\rVert}
\begin{document}
	\date{}	
	\title{Solution of Robust Linear Optimization Problems}
	\author[1]{Parthasarathi Mondal} \author[2]{Akshay Kumar Ojha}
	\affil[1]{\footnotesize Department of Mathematics, IIT Bhubaneswar, India\\ pm24@iitbbs.ac.in}
	\affil[2]{\footnotesize Department of Mathematics, IIT Bhubaneswar, India\\akojha@iitbbs.ac.in}
	\maketitle 
	
	\begin{abstract}
		Robust optimization(RO) is an important tool for handling optimization problem with uncertainty. The main objective of RO is to solve optimization problems due to uncertainty associated with constraints satisfying all realizations of uncertain values within a given uncertainty set. The challenge of RO is to reformulate the constraints so that the uncertain optimization problem is transformed into a tractable deterministic form. In this paper, we have given more emphasis to study the robust counterpart(RC) of the RO problems and have developed a mathematical model on the solution strategy for robust linear optimization problems, where the constraints only are associated with uncertainties. The box and ellipsoidal uncertainty sets are considered and some illustrative numerical examples have been solved in each corresponding case for validating our proposed method.
	\end{abstract}
	\vspace{1 mm} 
	
	\section{Introduction}
	Optimization plays an important role in the field of engineering, industry, medicine, business and almost all branches of science. Matheatical model of real life optimization problems very often associated with uncertain data. While considering, data can be inherently stochastic or random or it can be uncertain due to certain error associated with it. Some times, the error in the data may be incorrect estimation and due to lack of knowledge about the parameters of the mathematical model for uncertain demand in the particular inventory. Parameter uncertainty is a challenging task for solving optimization problem. In recent days, handling uncertainty in the optimization problem has drawn major focus in the mathematical programing community. During 1955, Dantzig \cite{1} put a foundation stone for solving optimization problem under uncertainty. Subsequently Charnes and Cooper \cite{2} striengthen the optimization techniques for solving stochastic programing and optimization under probabilistic constraints. In the solution of inventory problem, Scraf \cite{3} projected that the future demand will form a distribution that will differs from the governing past history in an unpredicted way and the majority of research work in decision making under uncertain data will be relied on the precise knowledge of probability. The solution of optimization problems under uncertainty remains a challenging task for the researcher for a quite long time. In the early 1970's Soyster \cite{4} first proposed a worst-case model for linear optimization considering the constraints are satisfying perturbations of the model parameters. Broadly, uncertain optimization can be dealt with two approaches such as stochastic optimization(SO) and robust optimization(RO). In SO, the uncertain optimzation is reformulated by considering the true probability distribution and then it is computationally tractable to solve. It gives only a probabilistic guarantee of optimal solution. More works on stochastic optimization can be found due to the work of Prekopa \cite{5}, Birge and Louveaux \cite{6}, and Shapiro \cite{7}. Impressed on the work of Soyester \cite{4}, where the column vectors of constraint matrix were constrained to belong to ellipsoidal uncertainty while solving uncertain problem, Falk \cite{8} finds the exact solution of inexact linear programing.\\
	
	Robust optimization on the other hand focused on more elaborated uncertainty set to addition the issues of over-conservatism in worst-used models, as well as to maintain computationally tractablity of the proposed method. The robust approach defends the uncertainty without using any sense of its probability distribution. In the mid 1990's, the research work due to Ben-Tal et al.\cite{9} have made a break through in the direction of RO. Subsequent works on RO by Ben-Tal and Nemirovski \cite{10,11,12,13} drwan much attention to the research community to work in this direction. Considering limited informations of underlying uncertainties such as mean and support, the model can provide a solution that is feasible to constraints with probability, although avoiding the extreme conservatism of Soyester's worst case model. Due to the advances of scientific computing and development of interior point method for convex optimization problems, particularly for semidefinite optimization by Boyed and Vandenberghe \cite{14} excelerate the interest the field of RO. The tractable uncertain set such as ellipsoid by Ben-Tal and Nemirovski \cite{11}, and Polytopes due to Bertsimas and Sim \cite{15} helps to achieve the computational tractability of robust linear constraints and robust counterparts of second order conic constraints respectively. The RO technique has a wide range of application in multi-period optimization in which future decision(recourse variable) depends on the realisation of present uncertainty. It is worth mentioning that such models are intractable. Ben-Tal et al. \cite{12} proposed a tractable approach for solving fixed recourse instances using  affine decision rules,  recouruse variables as affine functions of the uncertainty realization. A recent work on inventory model by Bertsimas, Thiele \cite{16} drawn attention to the researchers to work on the problem on uncertainty. As Robust linear optimization models are polynomial in size and in the form of linear programing(LP) or second order conic programming (SOCP) and requires only modest assumptions about distribution such as known mean and bounded support, it motivated the researchers to work in this direction. There have been many publications that show value of RO in many fields of applications including finance, energy, health care, scheduling and marketing etc. Indeed, the robust concepts and techniques are very useful as these are tailored to the information at hand and leads to tractable formulation.\\
	
	Many research works regarding robust optimization due to uncertainty have been discussed in many research papers, but very few of them have illustrated the details of solution approach with neumerical examples into account. In the present work we give a concise description of the basics of RO and related robust counterpart with suitable numerical examples. We introduce a new approach for the solution method that describes the descretisation technique of uncertainty set(region) of different shapes. In this approach, after selecting a nominal value within the uncertainty set, the set is descretised into infinitely many perturebations so as to cover the entire uncertainty set. This makes the sense of worst-case realization and gives the solution robust. Convergence idea ensures the desired robust optimal solution.\\
	
	Following the introduction, some important definitions, results, and perturbation formulation, and Ben-Tal's approaches of dealing RO problem are discussed in Section 2. Also the concept of robustness and the uncertain perturbations associated with RO have been discussed in this section. Section 3 describes the proposed work for solving RO problems. The solution method of RO problems have been discussed in this section with an algorithm and suitable numerical examples.  Finally, some conclusions related to our work have been incorporated in section 4.
	
	\section{Preliminary Discussions of Basic Concepts}
	Robust optimization (RO) is a new direction to solveoptimization problems uncertainties. Under this framework, the objective and constraint functions are belong to certain uncertainty sets. This approach aims to determine an optimal solution for the worst-case objective function. Robust optimal design is one with the best worst-case performance \cite{9}.
	% The general form of a robust linear optimization problem is,
	% \begin{align*}	
	% 	& &     &\min \quad f(x)\\
	% 	& &     &  ~\text{s.t.} \quad g_i(x) \leq 0, \quad i=1,2,\dots,n\\
	% 	& &     & ~~\qquad h_j(x)=0,  \quad j=1,2,\dots,m \tag{1.1}
	% \end{align*}
	% where $f, g_i, h_j$ are exactly known.
	% 
	% The above problem can be written in a matrix form as,
	% \begin{align*}
	% 	& &      &	\min \quad c^Tx\\
	% 	& &      &~\text{s.t.} \quad Ax \leq b \tag{1.2}
	% \end{align*}
	% where $x\in \mathbb{R}^n$ is the vector of decision variables, $c \in \mathbb{R}^n$ from the objective, $\textbf{A} \in \mathbb{R}^{m \times n}$, and $b\in \mathbb{R}^m$ is the right-hand side of the constraint. The collection ($c,A,b$) is the data set of the problem.\\
	
	In real life, the data in the objective and constraint functions in an RO are often not exactly known or at best known with some noise. The most common reasons for the data uncertainty could be due to the following conditions.
	\begin{enumerate}
		\item[(i)] Some of data entries do not exist when the data problem is solved and hence are replaced with their forcasts. These data entries are thus subject to prediction errors.
		\item[(ii)] Some of the data can not be measured exactly. In reality their values drift around the measured "nominal" values. These values are subject to measurement errors.
		\item[(iii)] Some of the decision variables can not be implemented exactly as calculated.
	\end{enumerate} 

\bigskip
	The general formulation of an uncertain linear optimization problem is given by,
	\begin{align*}
	\label{2.1}
	\{\min \{ c^Tx: Ax \geq b, x\in\mathbb{R}^n\}\}_{(c, A, b)}\in \mathcal{U} \tag{2.1}
	\end{align*}
	In	more precise form it looks, 
	\begin{align*}
	\label{2.2}
	&     &&\min \quad c^Tx \\
	&     &&~\text{s.t.} ~~~Ax \geq b, ~ (c,A,b)\in \mathcal{U}. \tag{2.2}
	\end{align*} 
	where, $c \in \mathbb{R}^n$, $A \in \mathbb{R}^{m \times n}$, $b \in \mathbb{R}^m$ represent the uncertain coefficients lying in $\mathcal{U}$ which denotes an uncertain set.
	
	\subsubsection{Some Basic assumptions:}
	We may assume without loss of generality that, (i) the objective is certain; (ii) the constraint right-hand side is certain; (iii) $\mathcal{U}$ is compact and convex; and (iv) the uncertainty is constraint-wise. We explain the reasons of why these four assumptions are not restrictive.
	\begin{enumerate}
		\item \textbf{The objective function is certain.}\\ 
		Suppose, the coefficients of the objective function ($\textbf{c}$) are uncertain and say $c \in \mathcal{C}$, where $\mathcal{C}$ is the uncertain set. Then \eqref{2.1} can be defined as follows,
		\begin{align*}
		\label{2.3}
		\min_x \max_{c \in \mathcal{C}} \quad [ c^Tx: Ax \geq b \quad \forall (A,b)\in \mathcal{U}] \tag{2.3}
		\end{align*}
		The problem \eqref{2.3} can be reformulated as follows,
		\begin{align*}
		\label{2.4}
		\min_{x,\alpha} \quad [ \alpha:  \alpha \geq c^Tx \quad \forall c \in \mathcal{C}, \quad Ax \geq b \quad \forall (A,b) \in \mathcal{U}] \tag{2.4}
		\end{align*}
		The above problem \eqref{2.4} is an uncertain problem in the variables $x$  and $\alpha$ but its objective function $\alpha$ is not at all affected by uncertainty. In the subsequent discussions, we will avoid the term 'uncertain' for the objective coefficient(c).
		\item  \textbf{The right-hand sides of the constraints are certain.}\\
		The uncertain right-hand side of a constraint can be translated into a certain coefficient by introducing an extra variable $x_{n+1}=-1$.
		\item\textbf{ The uncertainty set $\mathcal{U}$ is convex.}\\
		The RC of an RO problem remain unchanged if the uncertaity set $\mathcal{U}$ is replaced by its closed convex hull $conv(\mathcal{U})$, the smallest convex set containing $\mathcal{U}.$
		\item \textbf{The uncertainty is associated with constraints.}\\
		The  formulation (2.2) of RO problem can be written constraint-wise as, 
		\begin{align*}
		& &     &\min \quad c^Tx\\
		& &     &\text{s.t.} ~\quad a_i^Tx \geq b_i, ~\forall ~a_i \in \mathcal{U}_{a_i},~\forall~ b_i \in \mathcal{U}_{b_i},~ i=1,2,\dots,m.
		\end{align*}	
		where, $a_i$ represents the $i^{th}$ row of the uncertain matrix $A$, and $\mathcal{U}_{a_i} \subset \mathbb{R}^n$,   $\mathcal{U}_{b_i} \subset \mathbb{R}$ are the given uncertainty sets for $i=1,2,\dots,m$.
	\end{enumerate}
	
	\subsection{Important results and definitions}
	
	Some important definitions and results are incorporated below to illustrate our proposed work.
	%\begin{align*}
	%	& &     &\min \quad c^Tx\\
	%	& &     &~\text{s.t.} \quad Ax \leq b, \qquad (c,A,b)\in \mathcal{U}.
	%\end{align*}
	\begin{definition}
		A solution $x$ to the RO problem \eqref{2.2} is called \textbf{\emph{robust feasible}}, if it satisfies all the contraints $Ax \geq b$ for all realizations of the uncertain data i.e., for all $(c,A,b)\in \mathcal{U}.$
	\end{definition}
	
	\begin{definition}[Robust value]
		The robust value $\hat{c}(x)$ of the objective function in \eqref{2.2} is the largest value of the "true" objective $c^Tx$ over all realizations of the uncertain data.
		and is defined by, $\hat{c}(x)=\sup_{(c,A,b)\in\mathcal{U}}[c^Tx].$
	\end{definition}
	
	\subsubsection{Robust counterpart(RC)}
	\begin{definition}[]
		The RC of the uncertain problem \eqref{2.2} is an optimization problem of minimizing the robust value of the objective functions over all the robust feasible solutions to the uncertain problem which is defined as follows.
		\begin{equation*}
		\label{2.5}
		\min \left\{\hat{c}(x)=\sup_{(c,A,b)\in \mathcal{U}} c^Tx: Ax\geq b, \quad \forall(c,A,b)\in \mathcal{U}\right\}  \tag{2.5}
		\end{equation*}
	\end{definition}
	\vspace{2mm}
	In a more easier form, the above RC can be written as,
	\begin{align*}
	\label{2.6}
	& &     &\min \quad c^Tx\\
	& &     &\text{s.t.} ~\quad Ax \geq b, ~~ \forall(c,A,b)\in \mathcal{U}. \tag{2.6}
	\end{align*}
	An optimal solution of the RC \eqref{2.6} is a robust optimal solution of the RO problem \eqref{2.2} and the optimal value of the  RC is same as the robust optimal value of \eqref{2.2}.\\
	
	In view of Assumption1, if we assume the objective function is certain, the RC of the original problem \eqref{2.2} can be given as,
	\begin{equation*}
	\label{2.7}
	\min \left\{c^Tx: x\in\mathbb{R}^n, Ax\geq b, ~ \forall(A,b)\in \mathcal{U}\right\} \tag{2.7}
	\end{equation*}
	It is worth to mention that the uncertainty set is now a set in the space of constraint data i.e.,the robust counterpart of an RO problem with no uncertainty in the objective function is purely \textit{'constraint-wise'}.
	The following procedure is incorporated in order to construct the RC of an an RO problem. ~\cite{9}
	\begin{enumerate}
		\item[(i)] Keep the original certain objective function as it is.
		\item[(ii)] Replace each of the original constraints $(Ax)_i \geq b_i\Longleftrightarrow a_i^Tx\ge b_i$
		%	\begin{center}
		%		$(Ax)_i \le b\Longleftrightarrow a_i^Tx\le b_i$
		%	\end{center}
		with its corresponding counterpart $ a_i^Tx\geq b_i, ~ \forall a_i \in\mathcal{U}_{a_i}, b_i \in\mathcal{U}_{b_i}.$
		%	\begin{center}
		%		$ a_i^Tx\leq b_i \quad \forall a_i \in\mathcal{U}_{a_i}, b_i \in\mathcal{U}_{b_i}.$
		%	\end{center}
	\end{enumerate}

	\subsection{Uncertain set and perturbation}
	We mainly focus on solving the linear optimization problems associated with uncertainty. Consider a robust linear optimization problem with no uncertainty in the objective function \cite{9}.
	\begin{align*}
	\label{2.8}
	&   &&\min: \qquad 2x_1+3x_2\\
	&   &&~~\text{s.t.}  \qquad a_{11}x_1+a_{12}x_2 \geq b_1 \\
	&   &&\qquad\qquad a_{21}x_1+a_{22}x_2 \geq b_2 \tag{2.8}
	\end{align*}
	Let the constraints and the R.H.S coefficients are uncertain on the uncertainty set $\mathcal{U},$ where an element of $\mathcal{U}$ is:
	$\begin{bmatrix}  a_{11}  & a_{12} & b_1\\ a_{21} & a_{22} & b_2\end{bmatrix}=\begin{bmatrix}  a_{11}  & a_{12} \\ a_{21} & a_{22}\end{bmatrix}
	\begin{bmatrix} b_1 \\ b_2   \end{bmatrix}\in \mathcal{U} \subset \mathbb{R}^{2 \times 2} \times \mathbb{R}^{2 \times 1}$
	
	Now let $u=\begin{bmatrix}  1  & 2 & 1\\ 4 & 1 & 2\end{bmatrix}$ be a nominal element of $\mathcal{U}$ and the other elements are obtained from this nominal element $u$ by uniformly adding or subtracting 0.5 to each component. Then we get a finite number of elements of $\mathcal{U}$ \cite{9}.\\
	Now consider the general case where any component can be perturbed by adding or subtracting any amount from 0 up to and including 0.5. Using the nominal element $u$ we can construct an uncertainty set $\mathcal{U}$ with an infinite number of elements as follows.
	\begin{align*}
	\mathcal{U}=\left\{
	\begin{bmatrix}  a_{11}  & a_{12} & b_1\\ a_{21} & a_{22} & b_2\end{bmatrix}=\begin{bmatrix}  1  & 2 & 1\\ 4 & 1 & 2\end{bmatrix}+\sum_{l=1}^6 \xi_l P_l
	\right\}
	\end{align*}
	%where
	\begin{align*}
	& &     &P_1=\begin{bmatrix}  0.5  & 0 & 0\\ 0 & 0 & 0\end{bmatrix} \quad P_2=\begin{bmatrix}  0  & 0.5 & 0\\ 0 & 0 & 0\end{bmatrix} \quad P_3=\begin{bmatrix}  0  & 0 & 0.5\\ 0 & 0 & 0\end{bmatrix}\\
	& &      &P_4=\begin{bmatrix}  0  & 0 & 0\\ 0.5 & 0 & 0\end{bmatrix} \quad P_5=\begin{bmatrix}  0  & 0 & 0\\ 0.5 & 0 & 0\end{bmatrix} \quad P_6=\begin{bmatrix}  0  & 0 & 0\\ 0 & 0 & 0.5\end{bmatrix}
	\end{align*}
	and $\xi \in\mathcal{Z} = \{\xi=(\xi_1,\xi_2,\dots \xi_6)\in \mathbb{R}^6:-1\leq\xi_l\leq1, ~l= 1, \dots ,6\}$. The set $\mathcal{Z}$ is called the perturbation set. Corresponding to each realization of $\xi_l$ in the interval [-1,1], there is an element in $\mathcal{U}$.\\
	The matrices $P_l$ indicate that the $l$-th component in a $u\in \mathcal{U}$ is to be considered uncertain. 
	For example $P_1$ indicates that parameter $a_{11}$ can be perturbed by an amount between -0.5 to 0.5. Thus, the corresponding sets $\mathcal{U}_i$ can be written as,
	\begin{align*}
	\mathcal{U}_1=\{(a_1,b_1)=(1,2,1)+\sum_{l=1}^3 \xi_1^lP_1^l\quad :\xi_1 \in \mathcal{Z}_1\}
	\end{align*} and
	\begin{align*}
	\mathcal{U}_2=\{(a_2,b_2)=(4,1,2)+\sum_{l=1}^3 \xi_2^lP_1^l\quad :\xi_2 \in \mathcal{Z}_2\}
	\end{align*}
	where
	\begin{align*}
	& &       &	P_1^1=(0.5,0,0),\quad P_1^2=(0,0.5,0),\quad P_1^3=(0,0,0.5),\\
	& &       &  \mathcal{Z}_1=\{\xi=(\xi_1,\xi_2,\xi_3)\in \mathbb{R}^3|-1\leq\xi_l\leq1, ~ l= 1, \dots ,3\}.\\
	& &       &	P_2^1=(0.5,0,0),\quad P_2^2=(0,0.5,0),\quad P_2^3=(0,0,0.5),\\
	\text{and}\\
	& &       &  \mathcal{Z}_2=\{\xi=(\xi_1,\xi_2,\xi_3)\in \mathbb{R}^3|-1\leq\xi_l\leq1, ~ l= 1, \dots ,3\}.
	\end{align*}
	In general, for the constraint $a_i^Tx\leq b_i; ~ [a_i;b_i]\in \mathcal{U}_i$ in an RO problem, the uncertainty set $\mathcal{U}_i$ can be written as,
	\begin{align*}
		\label{2.9}
	\mathcal{U}_i=\{[a_i;b_i]=[a_i^0;b^0]+\sum_{l=1}^{L_i} \xi_i^lP_i^l\quad:\xi \in \mathcal{Z}_i\} \tag{2.9}
	\end{align*}
	Then the constraint in the corresponding robust counterpart is,
	\begin{align*}
	a_i^Tx\leq b_i ~~ \forall \{[a_i;b_i]=[a_i^0;b_i^0]+\sum_{l=1}^{L_i} \xi_i^lP_i^l\quad :\xi \in \mathcal{Z}_i\}
	\end{align*}
	where $(a_i^0,b_i^0)$ is the nominal value of $(a_i,b_i)$, $\mathcal{Z}_i$ is the perturbation set, and ${L_i}$ is the number of elements in $(a_i,b_i)$ that are to be considered uncertain. \\
	
	We can write this in a more generalized form as
	\begin{align*}
	\label{2.10}
	\mathcal{U}=\{[a;b]=[a^0;b^0]+\sum_{l=1}^L \xi_lP_l \quad:\xi \in \mathcal{Z}\} \tag{2.10}
	\end{align*}
	When $\mathcal{U}$ is infinte, the tractability of the corresponding RC will depend on the structure of $\mathcal{U}$. In particular, the structure of the corresponding perturbation set characterizes the tractability.\\
	
	\subsection{Ben-Tal and Nemirovski approach to robust optimization \cite{9}} 
	Consider the linear optization program with no uncertainty in the right-hand side $b_i$ of the constraints as, 
	\begin{align*}
	& &     &\max \quad c^Tx\\
	& &     &\text{s.t.} \quad~  a_i^Tx \geq bi, ~\quad i=1,2,\dots,m
	\end{align*}
	
	Assume each row $a_i$ of A is uncertain which lies in the ellipsoidal uncertainty sets $\mathcal{U}_{a_i}$, where 
	\begin{center}
		$\mathcal{U}_{a_i}=\left\{a_i={a_i}^0+\sum_{l=1}^{L_i}P_l \xi_l \quad:\norm{\xi_l}_2 \geq 1 \right\} \quad \forall i=1,2,\dots,m$
	\end{center}
	Worst-case realization forces to reformulate the constraints as,
	\begin{equation*}
	\begin{aligned}
	& \underset{a_i}{\text{max}}	
	&  a_i^Tx \geq b_i,~ \forall i=1,2,\dots,m.
	\end{aligned}
	\end{equation*}
	The maximum in the interior makes the problem easier. Indeed,
	\begin{align*}	
	\max[a_i^Tx: a_i \in \mathcal{U}_{a_i}] =[{a_i}^0]^Tx+\max\left\{\sum_{l=1}^{L_i}\xi_l^TP_i^Tx, \quad: \norm{\xi_l}_2 \geq 1\right\}
	\end{align*}
	Now use the fact that $\underset{\norm{u_i}_2 \leq 1}{\text{max}}\sum_{l=1}^{L_i}\xi_l^TP_i^Tx=\norm{P_i^Tx}_2$. 
	This equality is due to Cauchy-Schawrz applied to $\xi$ and $P_i^Tx.$ Therefore, the RC is obtained in the following form,
	\begin{align*}
	\label{2.11}
	& &     &	\max \quad c^Tx\\
	& &     &	\text{s.t.} \quad~  a_i^Tx+\sum_{l=1}^{L_i}\norm{P_i^Tx}_2 \geq bi,\quad i=1,2,\dots,m. \tag{2.11}
	\end{align*}
	This is a second-order cone programming problem (SOCP)\cite{9}.
	
	In this context, the general form of a second order conic problem can be given as follows,
	\begin{align*}
	& &     &\min \qquad \dfrac{1}{2}c^Tx\\
	& &     &\text{s.t} ~~~~\quad \norm{Ax-b}_2 \geq f^tx+g
	\end{align*} where  $A$ is a $m \times n$ matrix, $f$, $g$ are functions, and the remaining quantities have conformable dimensions.
	The SOCP is usually obtained as a result of incorporating robustness into linear programme..

	\section{Proposed approach and methodology}
	Notice that \eqref{2.2} contains infinitely many constraints due to the for all ($\forall$) quantifier imposed by the worst case formulation, i.e., it seems intractable in its current form. There are two ways to deal with this. The first way is to apply robust reformulation techniques to exclude the for all ($\forall$) quantifier. If deriving such a robust reformulation is not possible, then the second way is to apply the discretisation approach. In our work, we describe the details of these two approaches.
	
	\subsubsection{Methodologies for solving a robust linear optimization problem}
	Before going to solve the robust linear optimization problems, we give a methodology for solving such problems. The following steps are adopted to obtain the optimal solution of a robust linear optimization problem \eqref{2.2}.
	
	\textbf{Proposed Algorithm}\\
	\textbf{Step1:} Obtain the tractable form of the corresponding robust counterpart of the considered robust optimisation problem.\\
	\textbf{Step2:} Find out the uncertain coefficients corresponding to the constraints and objective function and set an arbitrary nominal value for all those uncertain values within the specified uncertainty set.\\
	\textbf{Step3:} Using MATLAB or MATHEMATICA or any other software available for the same, take as many grids of the uncertain region as possible so as to increase the number of realised values and to realise the worst-case phenomenon of the uncertainty. Then Make a table of the optimal solutions according to the realised values.\\
	\textbf{Step4:} Finally, find out the optimum solution from the table correct up to the desired choice of decimal place.
	
	We give a theorem that guarantees the feasibility of the solutions of an RO problem and its RC.  Consider a robust linear optimization problem of the form \eqref{2.1}, 
	\begin{align*}
	\label{P}
	\{\min_x \{ c^Tx: Ax \geq b, x\in\mathbb{R}^n\}\}_{(c, A, b)}\in \mathcal{U} \tag{P}
	\end{align*} 
	and its RC \eqref{2.7} as
	\begin{equation*}
	\label{P*}
	\min_x \left\{c^Tx: x\in\mathbb{R}^n, Ax\geq b, ~ \forall(A,b)\in \mathcal{U}\right\} \tag{P*}
	\end{equation*}
	
	\begin{theorem}
		\label{th1}
		Let \eqref{P} be an RO optimization problem with constraint-wise uncertainty. Then the RC of \eqref{P} is feasible if and only if all the instances of \eqref{P} are feasible, and its robust optimal value is given by, ${\sup c^Tx}_{(A,b)\in\mathcal{U}}$
	\end{theorem}
	
	\begin{proof}
		First we show that, if all the instances are feasible, then the RC is feasible.
		On the contray, let \eqref{P*} is infeasible. Then the family of sets, $$\{S_i(a_i)=\{x\in\mathbb{R}^n:a_i^Tx\geq b_i\}\}_{i=1,2,\dots, m,~ a_i\in\mathcal{U}_{a_i}, b_i\in\mathcal{U}_{b_i}}$$ has an empty intersection.
		Since, for each $i=1,2,\dots, m, S_i(a_i)$ is a closed subset of the compact set $\mathbb{R}^n$, there exists a collection $a_{i,j}\in\mathcal{U}_{a_i}, b_{i,j}\in\mathcal{U}_{b_i}, i=1,2,\dots,m,~ j=1,2,\dots, M$ such that,
		$$\cap_{i\leq m, j\leq M}S_i(a_{i,j})=\phi$$
		This follows obviously that, 
		$$\max_{x\in\mathbb{R}^n} \min_{i\leq m, j\leq M} a_{i,j}^Tx<b_{i,j}$$
		i.e., $$\max_{x\in\mathbb{R}^n} \min_{i\leq m, j\leq M} (a_{i,j}^Tx-b_{i,j})<0$$.
		Since, the set $\mathbb{R}^n$ is compact and all the constraints are concave, it follows that there exists a convex combination, $$a^Tx=\sum_{i\leq m, j\leq M}\lambda_{i,j}(a_{i,j}^Tx-b_{i,j})$$ of the constraints $(a_{i,j}^Tx-b_{i,j})$ which is strictly negative.\\
		Set, $\lambda_i=\sum_{j=1}^{M}$, and $a_i=\lambda_i^-1\sum_{i=1}^{m}\lambda_{i,j}a_{i,j}.$\\
		Then we have, $$\lambda_ia_i^Tx=\sum_{i=1}^{M}\lambda_{i,j}a_{i,j}$$
		i.e.,  $$(\lambda_ia_i^Tx-b_{i,j})=\sum_{i=1}^{M}\lambda_{i,j}(a_{i,j}-b_{i,j})$$. This follows, 
		\begin{align*}
		\label{3.1.1}
		\sum_{i=1}^{m}(\lambda_ia_i^Tx-b_{i,j})=\sum_{i=1}^{m}\sum_{i=1}^{M}\lambda_{i,j}(a_{i,j}-b_{i,j}) \tag{3.1.1}
		\end{align*}
		Since, for all $i=1,2,\dots m$ and $j=1,2,\dots M$, $(a_{i,j}-b_{i,j})$ is strictly positive, it can be easily say that, the expression in \eqref{3.1.1} is strictly positive. i.e., 
		\begin{align*}
		\label{3.2.1}
		\sum_{i=1}^{m}(\lambda_ia_i^Tx-b_{i,j})=\sum_{i=1}^{m}\sum_{i=1}^{M}\lambda_{i,j}(a_{i,j}-b_{i,j}) < 0. \tag{3.2.1}
		\end{align*}
		Since the uncertainty is constraint-wise and $a_i\in \mathcal{U}_{a_i}; b_i\in\mathcal{U}_{b_i}$ are convex for all $i=1,2,\dots m$, the point $A=(a_1, a_2, \dots, a_m)\in\mathcal{U}=\mathcal{U}_{a_1}\times \mathcal{U}_{a_2}\dots \times \mathcal{U}_{a_m}$ and $b_i\in \mathcal{U}_{b_i} $.\\
		Thus, the expression \eqref{3.2.1} shows that the instances corresponding to the uncertain data $(A,b)\in\mathcal{U}$ are infeasible, which is a contradiction to our assumption.\\
		The converse part is obvious from the definition of RC.
	\end{proof}
	
	\subsection{Robust reformulation approach}
	%Gaining concepts of tractibility of robust linear optimization problems under several types of uncertainty sets, we find interest to solve certain type of robust linear optimization problems. The robusrt counterpart of \eqref{2.2} with interval uncertainty set and ellipsoidal uncertainty set have been improved.\\
	For the robust reformulation technique, we use the problem \eqref{2.2} in the following form,
	\begin{align*}
	\label{3.1}
	&   &&\min \quad~~ c_1x_1+c_2x_2+\dots +c_nx_n\\
	&   &&~~\text{s.t.} \quad ~~a_{11}x_1+a_{12}x_2 + \dots +a_{1n}x_n \geq b_1 \\
	&   &&\qquad\qquad	a_{21}x_1+a_{22}x_2 + \dots +a_{2n}x_n  \geq b_2\\
	&   &&\qquad\qquad \vdots\\
	&   &&\qquad\qquad a_{m1}x_1+a_{m2}x_2 + \dots +a_{mn}x_n  \geq b_n	 \tag{3.1}
	\end{align*} 
	where $c_i, a_{ij},b_i;~ i=1,\dots, n, j=1,\dots, m,$ are uncertain data in a specified uncertainty set $\mathcal{U}.$ We give the solution methods for solving this type of optimization problem having number of constraints with uncertainty in the constraint coefficients, where the uncertain data come from a known mathematical figure such as box, paraboloid, ellipsoid, etc.
	
	\subsubsection{Uncertain RO problem with single uncertainty-affected constraint}
	Let us consider a robust linear problem, with a finite uncertainty set, where uncertainty is in parameter values. The corresponding RC is therefore linear and computationally tractable. The following is a minimisation problem with only one uncertainty-affected constraint.
	\begin{align*}
	\label{3.2}
	&     &&\min \quad c_1x_1+c_2x_2\\
	&     &&	\text{s.t.} \quad~~ a_1x_1+a_2x_2 \geq b  \tag{3.2}
	\end{align*}
	Take the instance $c_1 =2, c_2=3, a_1=2, a_2=1$, and $b_1 =1$ for the the nominal problem. The solution of the corresponding linear optimization problem is $x_1=0.5, x_2=0$ and the objective value is 1.
	If we consider the the values for $a_1$ and $a_2$ are only estimates and can be inaccurate, and that the actual value that realize are $a_1=1.99, a_2=1.01$, then the optimal solution of the nominal problem is no longer feasible for this realization. This is problematic in situations where the the decision has to be taken here and now and the constraints are hard in the sence that they must be satisfied by all realizations of actual values of parameters.
	
	Now, the robust counterpart of the linear program \eqref{3.2}, with uncertain constraint coefficients $a_1$ and $a_2$  is,
	\begin{align*}
	& &     &\min \quad c_1x_1+c_2x_2\\
	& &     &\text{s.t.} \quad ~~a_1x_1+a_2x_2 \geq b \\
	& &     &~~~ \quad \quad	\forall~ (a_1,a_2)\in \mathcal{U}; ~\mathcal{U} \subset \mathbb{R}^{1 \times 2}
	\end{align*}
	For this problem a solution $x^T=(x_1,x_2)$ must satisfy the constraint $a_1x_1+a_2x_2 \geq b$ for all $(a_1,a_2)$ in $\mathcal{U}$.
	
	Suppose that $(a_1,a_2)$ has the three realizations i.e., $\mathcal{U}  =
	\left\{
	\begin{bmatrix} 1.99\\0.99 \end{bmatrix},
	\begin{bmatrix} 2.00\\1.00  \end{bmatrix},
	\begin{bmatrix} 2.01\\1.01  \end{bmatrix}
	\right\}$, then the robust counterpart of the nominal linear program consists of these three realizations as follows:
	\begin{align*}
	\label{3.3}
	& &    &\min \quad 2x_1+3x_2\\
	& &    &\text{s.t.} \quad 1.99x_1+0.99x_2 \geq 1 \\ 
	& &     &~~~~~~~~~~~2x_1+1x_2 \geq 1 \\
	& &     &~~~~~~~ 2.01x_1+1.01x_2 \geq 1 \tag{3.3}
	\end{align*}
	The result \eqref{3.3} is a computationally tractable RC of \eqref{3.2}, which contains a finite number of constraints.
	The optimal solution to the robust counterpart is $\bar x_1=0.5025$ and $\bar x_2=0.0000$ with a corresponding objective value 1.0050. The solution has the advantage of satisfying all the constraints without increasing the objective value too much. In this case the solution is robust or immune to uncertainty. So, by Theorem\ref{th1} the robust optimal somution of \eqref{3.2} is the same as that of the RC \eqref{3.3}, i.e., the robust solution of the robust linear optimization problem \eqref{3.2} is $\bar x_1=0.5025$ and $\bar x_2=0.0000$ with the optimal objective value $\bar c=1.0050.$. 
	
	\subsubsection{Uncertain RO problem with two uncertainty-affected constraints}
	Next we consider a minimisation problem with two uncertainty-affected constraints. The same procedure can be used for \eqref{3.1} to get a robust solution.
	\begin{align*}
	\label{3.4}
	& &     &\min \quad 2x_1+3x_2\\
	& &     &\text{s.t.} \quad a_{11}x_1+a_{12}x_2 \geq b_1 \\
	& &     & \qquad~ a_{21}x_1+a_{22}x_2 \geq b_2 \tag{3.4}
	\end{align*}
	The constraint coefficients  and right-hand coefficients are uncertain. The corresponding robust counterpart is,
	\begin{align*}
	& &    &\min \quad 2x_1+3x_2\\
	& &    &\text{s.t.} \quad a_{11}x_1+a_{12}x_2 \geq b_1 \\
	& &    &\qquad ~a_{21}x_1+a_{22}x_2 \geq b_2\\
	& &    &\forall	\left\{
	\begin{bmatrix}  a_{11}  & a_{12} \\ a_{21} & a_{22}\end{bmatrix}
	\begin{bmatrix} b_1 \\ b_2   \end{bmatrix}
	\right\} \in \mathcal{U} \subset \mathbb{R}^{2 \times 2} \times \mathbb{R}^{2 \times 1} 
	\end{align*}
	To avoid confusion only the first bracket can be used to write the element,
	\begin{center}
		$\begin{bmatrix}  a_{11}  & a_{12} \\ a_{21} & a_{22}\end{bmatrix}
		\begin{bmatrix} b_1 \\ b_2   \end{bmatrix}=\begin{bmatrix}  a_{11}  & a_{12} & b_1\\ a_{21} & a_{22} & b_2\end{bmatrix}$
	\end{center}
	Suppose that $(a_i,b_i)=(a_{i1}, a_{i2},b_i)$ has the three realizations of actual values.
	i.e., $\mathcal{U}  =
	\left\{
	\begin{bmatrix} 0.95 & 1.95 & 0.95\\2.95 & 1.95 & 1.95  \end{bmatrix},
	\begin{bmatrix} 1 & 2 & 1\\3 & 2 &2  \end{bmatrix},
	\begin{bmatrix} 1.05 & 2.05 & 1.05 \\3.05 & 2.05 & 2.05  \end{bmatrix}
	\right\}$,\\ 
	then the robust counterpart of the nominal linear program is the follwoing,
	\begin{align*}
	\label{3.5}
	& &     &\min \quad  z\\
	& &     &\text{s.t:} \quad z \geq 2x_1+3x_2\\
	& &     &\qquad \quad  a_{11}x_1+a_{12}x_2 \geq b_1 \\
	& &     &\qquad \quad a_{21}x_1+a_{22}x_2 \geq b_2 \tag{3.5}
	\end{align*}
	Here $\mathcal{U}$ being finite the robust counterpart also remains finite, linear, and computationally tractable. However if the set $\mathcal{U}$ has an infinite number of elements, the number of constraints will be infinite and we have the semi-infinite linear program, i.e., a linear program with an infinite number of constraints which is generally an intractable class of problems.\\
	The optimal solution to the robust counterpart is $\bar x_1=0.5$ and $\bar x_2=0.2561$ with its objective value $\bar z=1.7683$. Hence, by Theorem\ref{th1} the robust optimal somution of \eqref{3.4} is $\bar x_1=0.5$ and $\bar x_2=0.2561$ with the robust objective value with its objective value $\bar z=1.7683$.
	
	\subsection{Robust discretisation technique}
	For the discretisation technique, we use the general form of an uncertain liner optimization problem \eqref{2.2} in the following form,
	\begin{align*}
	\label{3.6}
	&   \underset{x}{\text{max}}	
	&    & f(x) \\
	&   &&~~\text{s.t.} \quad ~~u_{11}(\xi)x_1+u_{12}(\xi)x_2 + \dots +u_{1n}(\xi)x_n \leq b_1 \\
	&   &&\qquad\qquad	u_{21}(\xi)x_1+u_{22}(\xi)x_2 + \dots +u_{2n}(\xi)x_n  \leq b_2\\
	&   &&\qquad\qquad \vdots\\
	&   &&\qquad\qquad u_{m1}(\xi)x_1+u_{m2}(\xi)x_2 + \dots +u_{mn}(\xi)x_n  \leq b_n \tag{3.6}
	\end{align*}

	\subsubsection{Interval uncertainty} 
	Consider the following uncertain linear optimization problem.
	\begin{align*}
	\label{3.7}
	& {\text{max}}	
	& & f(x)=5x_1+3x_2+4x_3 \\
	&\text{s.t.} & &(1+\xi_1+2\xi_2)x_1 +(1-2\xi_+\xi_2)x_2+(2+2\xi_1)x_3 \leq 18\\
	&         & & (\xi_1 + \xi_2)x_1 + (1-2\xi_2)x_2 + (1-2\xi_1-\xi_2)x_3\leq 16 \qquad \forall x\in \text{Box} 	\tag{3.7}
	\end{align*}
	where $\text{Box}= \{\xi |-a \leq \xi_1 \leq a; ~-b \leq \xi_2\leq b: a,b\in\mathbb{N}\}$  is the given uncertainty
	set and $x_1,x_2,x_3$ are non-negative integer variables. This is a linear optimization prolem having two uncertain constraints.\\
	It is challengiable to solve this type of problem. Here objective function has no uncertainty. Also the right-hand side of the constraints are not uncertainty affected. The only uncertainty is in the constraints and completely lies in the rectangular region of dimension $2a$ and $2b$ i.e., the region bounded by the two intervals $\{\xi_1:-a \leq \xi_1 \leq a\}$ and $\{\xi_2: -b \leq \xi_2\leq b\}~ \text{for}~ a,b\in\mathbb{N}$.\\
	Increasing the number of grid points in the box while solving the problem the box can be mostly covered and we get a finite number of linear inequalities.  See \figurename{\ref{b1}}.

	\begin{figure}[H]
		\centering
		\subcaptionbox{Small number of grids of box}[0.50\textwidth]{\includegraphics[width=0.50\textwidth]{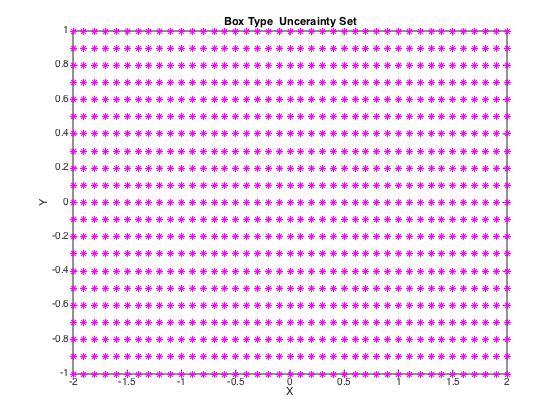}}%
		\hfill % <-- Seperation
		\subcaptionbox{Large number of grids of box}[0.50\textwidth]{\includegraphics[width=0.50\textwidth]{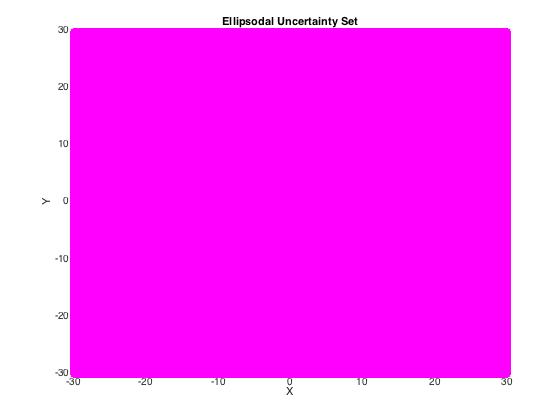}}%
		\\ % <-- Line break
		\caption{Grid points of box region}
		\label{b1} 
	\end{figure}
	
	In particular, if we take $a=1,b=1$, under the uncerainty set of unit square box the optimal solution to this problem is $x_1=5, x_2=2.5, x_3=0
	$ and the corresponding optimal objective value is 25. For $a=2,b=1$ the optimal solution is $x_1=4.744, x_2= 2.179, x_3=0$ and the corresponding optimal objective value is 23.333. The optimal solutions that we get, satisfy all the constraints for all realizations of data and therefore the solution is \textbf{robust}.
	
\begin{table}[h!]
		\caption{\bf{Robust optimal solutions under box uncertainty set}}
	\label{tab:table1}
	\centering
	\begin{tabular}{||c c c c||} 
		\hline
		\textbf{Uncertain Region} & \textbf{Optimal Solution} & \textbf{Optimal Value} & \textbf{No. of Constraints}\\
	${(-a\leq\xi_1\leq a;-b\leq\xi_2\leq b )}$& {($x_1,x_2,x_3)$} & {$(f_{max})$} & {} \\
		\hline\hline
		$a=1,b=1$ &  {(4.21.0, 2.32, 0.2)} & 28.810 & 92 \\
	$a=2,b=1$ &  {(4.17, 2.18,  0.19)} & 28.151 & 124\\
	$a=2,b=2$ & {(3.34, 1.67,  0.13)} & 23.760 & 875\\
%	\vspace{.3cm}
	$a=5,b=5$ & {(2.721, 1.341,  0.840)} & 20.82& 2524\\
	$a=10,b=5$ &{(2.654, 1.312, 0.831)} & 20.38 & 2578\\
	$a=10,b=10$ & {(2.344, 1.252,0.721)} & 18.25 & 6874\\
	$a=20,b=10$ & {( 2.138,  1.194, 0.616)} & 14.349 & 11245\\
	$a=20,b=20$ & {(2.113, 1.205, 0.326)} & 12.315 & 18395\\
	$a=30,b=20$ & {(2.107, 1.72,0.325)} & 12.302 & 25257\\
	$a=30,b=30$ & {(2.088, 0.91,0.221)} & 11.573& 45107 \\ [1ex] 
		\hline
	\end{tabular}
\end{table}

	\subsubsection{Result Analysis}
	For the validity of our work, we have incorporated a benchmarked problem \eqref{3.8} which is taken from \cite{17}. We have solved this problem using our proposed method. It is observed that the results in our method are very closed to the results of the problem in \cite{17}. The closedness of our result is shown in Table\ref{tab:table2}. For this, we recall the problem stated in \cite{17}.
	
	\begin{align*}
	\label{3.8}
	& {\text{max}}	
	& & f(x)=5w+3z_1+4z_2 \\
	&\text{s.t.} & &(1+\xi_1+2\xi_2)w +(1-2\xi_+\xi_2)z_1+(2+2\xi_1)z_2 \leq 18\\
	&         & & (\xi_1 + \xi_2)w + (1-2\xi_2)z_1 + (1-2\xi_1-\xi_2)z_2\leq 16 \tag{3.8}
	\end{align*}
	
	The Table\ref{tab:table2} presents the optimal results of the method in \cite{17} and the results obtained from our proposed method.  

\begin{table}[h!]
	\caption{\bf{Comparison of the results by two methods}}
\label{tab:table2}
	\centering
	\begin{tabular}{||c c c c c||} 
		\hline
		\textbf{Uncertain Region} & \textbf{OptSol} &\textbf{OptSol} & \textbf{OptVal} & \textbf{OptVal}\\
	${(-a\leq\xi_1\leq a;-b\leq\xi_2\leq b )}$& {($w,z_1,z_2)$} & {($x_1,x_2,x_3)$} &{$(z_{max})$} & {$(f_{max})$}\\
		\hline\hline
		$a=1,b=1$ & {(1, 4,  3)} & {(4.210, 2.32,0.2)} & 29 & 28.81 \\
	$a=2,b=1$ & {(1, 4.21, 2.97)} &  {(4.17, 2.18,  0.19)} & 29.51 & 28.15\\
	$a=2,b=2$ & {(0.9, 3.96, 2.83)} & {(3.34, 1.67,  0.13)} & 27.7 & 23.76\\
	$a=5,b=5$ & {(0.93, 3.89,2.78)} &  {(2.72, 1.34,  0.8)} &27.44 & 20.82\\
	$a=10,b=5$ & {(0.91, 3.81, 2.67)} &  {(2.65, 1.31, 0.8)} &26.66 & 20.38\\
	$a=10,b=10$ & {(0.9, 3.73, 2.87)} &  {(2.34, 1.25,0.7)} &27.17 & 18.25\\ [1ex] 
		\hline
	\end{tabular}
\end{table}

	\subsubsection{Ellipsoidal uncertainty}
	When uncertainty lies in ellipsoid the general form of ellipsoidal uncertainty set is of the form, $\frac{\xi_1^2}{a^2}+\frac{\xi_2^2}{b^2}+\frac{\xi_3^2}{c^2} =1$. In this case, the values of $\xi_1,\xi_2,\xi_3$ are not exactly known, but they together satisfy the ellipsoidal behaviour. This ellipsoid can be subdivided to find the grid points.
	
	For the ellipsoidal uncertainty case, we consider the same problem \eqref{3.7} with ellipse as uncertainty set. The problem can be written as,
	\begin{align*}
	\label{3.9}
	&{\text{max}}	
	& & f(x)=5x_1+3x_2+4x_3 \\
	&\text{s.t} & &(1+\xi_1+2\xi_2)x_1 +(1-2\xi_+\xi_2)x_2+(2+2\xi_1)x_3 \leq 18\\
	&         & & (\xi_1 + \xi_2)x_1 + (1-2\xi_2)x_2 + (1-2\xi_1-\xi_2)x_3\leq 16~ \forall x\in \text{Ellipse} \tag{3.9}
	\end{align*}
	where $\mathcal{U}$  is the given ellipsoidal uncertainty set and $x_1,x_2,x_3$ are non-negative integer variables. The uncertain parameters $\xi_1, \xi_2$ come from an ellipse $\frac{\xi_1^2}{a^2}+\frac{\xi_2^2}{b^2} =1$. By changing the elliptic region i.e. for various value of $a$ and $b$ the problem can be solved. Solving of this linear problem satisfies all the realizations of uncertain parameters values is $x_1=4.210, x_2= 2.32, x_3= 0.2$ and the corresponding maximum objective value is 28.810.\\
	
	\begin{figure}[H] 	\label{b2} 
		\centering
		\subcaptionbox{Small number of grids ofellipse}[0.50\textwidth]{\includegraphics[width=0.45\textwidth]{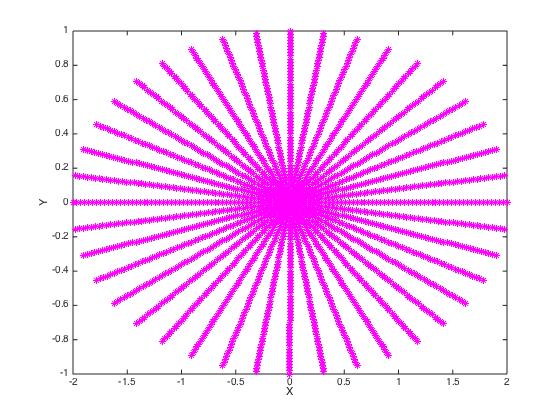}}%
		\hfill % <-- Seperation
		\subcaptionbox{Large number of grids of ellipse}[0.50\textwidth]{\includegraphics[width=0.45\textwidth]{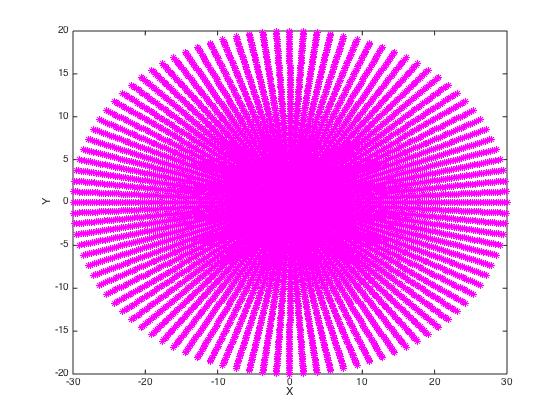}}%
		\\ % <-- Line break
		\caption{Grid points of ellipsoidal region}
	\end{figure}

The Table\ref{tab:table3} shows the opimal result of \eqref{3.9}. 	However, the uncertainty can be considered in the generaliged version of the ellipsoid.
\begin{table}[h!]
	\caption{\bf{Robust optimal solutions under ellipsoidal uncertainty set}}
	\label{tab:table3}
	\centering
	\begin{tabular}{||c c c c||} 
		\hline
		\textbf{Uncertain Region} & \textbf{Optimal Solution} & \textbf{Optimal Value} & \textbf{No. of constraints}\\
		${(\frac{\xi_1^2}{a^2}+\frac{\xi_2^2}{b^2} =1)}$& {($x_1,x_2,x_3)$} & {$(f_{max})$} & {} \\
		\hline\hline
		$a=1,b=1$ & {(4.7072, 3.9650, 0.000)} & 26.7585 & 1477 \\
		$a=2,b=1$ & {(4.7968, 2.5783, 0.1219)} & 24.9534 & 983\\
		$a=2,b=2$ & {(3.3348, 2.7156, 0.000)} & 18.7703 & 1590\\
		\vspace{.3cm}
		$a=5,b=5$ & {(1.2675, 1.1996, 0.3307)} &9.1225 & 1437\\
		$a=10,b=5$ & {(0.4001, 0.4001, 0.8873)} & 8.2948 & 1016\\
		$a=10,b=10$ & {(0.5865, 0.6486, 0.2622)} &4.9542 & 1184\\
		$a=20,b=10$ & {(0.1766,  0.6225,  0.5279} & 4.5912 & 950\\
		$a=20,b=20$ & {(0.2830, 0.3386, 0.1571)} & 2.5947 & 1146\\
		$a=30,b=20$ & {(0.1533, 0.3278,0.2408)} & 2.4730 & 721\\
		$a=30,b=30$ & {(0.1760,  0.2312,  0.1184)} &  1.7586 & 809\\ [1ex] 
		\hline
	\end{tabular}
\end{table}		
	
	\subsubsection{Comparative result analysis of box and elliptic uncertainty sets}
	For a particular robust linear maximization problem with certain objective function and certain right hand sides in the constraints, we compare the optimal solutions and optimal objective values of \eqref{3.7} an \eqref{3.9} under the box and ellispsoidal uncertainty sets respectively. Several number of observations help to clear the comparison decision.
	
\begin{table}[h!]
	\caption{\bf{Robust optimal solutions under box and elliptic uncertainty sets}}
	\label{tab:table4}
	\centering
	\begin{tabular}{||c c c c c||} 
		\hline
		\textbf{Uncert Region} & \textbf{OptSol(Box)} & \textbf{OptVal(Box)} & \textbf{OptSol(El)} & \textbf{OptVal(El)}\\
		${(\xi_1, \xi_2 )}$& {($x_1,x_2,x_3)$} & {$(f_{max})$} & {($x_1,x_2,x_3)$} & {$(f_{max})$} \\
		\hline\hline
		$a=1,b=1$ &  {(4.210, 2.322, 0.235)} & 28.810 & {(4.707, 3.965, 0.000)} & 26.758 \\
		$a=2,b=1$ &  {(4.173, 2.184,  0.191)} & 28.151 & {(4.797, 2.578, 0.122)} & 24.953\\
		$a=2,b=2$ & {(3.344, 1.672,  0.134)} & 23.760 & {(3.335, 2.715, 0.00)} & 18.770 \\
		$a=5,b=5$ & {(2.721, 1.341, 0.840)} & 20.82& {(1.267, 1.199, 0.338)} &9.122\\
		$a=10,b=5$ &{(2.654, 1.312, 0.831)} & 20.38 & {(0.400, 0.400, 0.887)} & 8.295\\
		$a=10,b=10$ & {(2.344, 1.252,0.721)} & 18.25 & {(0.586, 0.649, 0.262)} &4.954\\
		$a=20,b=10$ & {(2.138,  1.194, 0.616)} & 14.349 & {(0.176,  0.622,  0.528} & 4.591\\
		$a=20,b=20$ & {(2.113, 1.205, 0.326)} & 12.315 & {(0.283, 0.338, 0.157)} & 2.595\\
		$a=30,b=20$ & {(2.107, 1.721, 0.325)} & 12.302 & {(0.153, 0.328,0.249)} & 2.473 \\
		$a=30,b=30$ & {(2.088, 0.913, 0.221)} & 11.573& {(0.176,  0.231,  0.118)} &  1.758 \\ [1ex] 
		\hline
	\end{tabular}
\end{table}	

	For the RO problem with the box and ellipsoidal uncertainty sets, the comparative results are shown in Table\ref{tab:table4}. When the uncertainty set is box, the maximum objective value of \eqref{3.7} reduces for bigger values of $a$ and $b$. The bigger the uncertain region the smaller the optimal objective value. In the other words, size of uncertain region affects the maximum values of the problem. The same scenario happens for  \eqref{3.9} under the ellipsoidal uncertainty set. The bigger values of major and minor axes reduce the maximum value of the problem. But, the maximum objective value do not cross the positivity due to non-negativity conditions of decision variables and positive coefficients of objective function. 
	
	\begin{figure}
		\begin{center}
			\label{fig:fig3}
			\includegraphics[width=.7\textwidth]{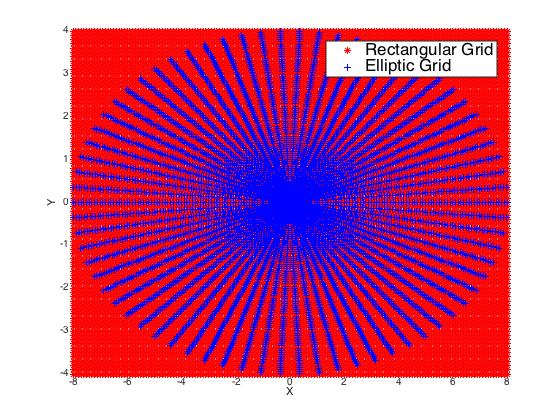}
			\caption{Simultaneous grid points of box and ellipse}
			\label{fig:3}
		\end{center}
	\end{figure}
	\par Comparing the robust optimal solutions and optimal objective values under the box uncertanty set and ellipsoidal uncertainty set respectively in Table\ref{tab:table4}, it is observed that under a certain dimension of uncertainty set the ellpsoidal set gives better result for a robust linear maximization problem. The reason of such behaviour is due to the size of uncertaity sets. For a certain dimension, box contains more number of grid points than that of an elliptic region [see Fig \ref{fig:3}] and therefore it gives more number of constraints in the RC than that of the ellipse. 
	In each case, the solutions satisfy all the constraints for all realizations of the uncertain data and therefore the solution are robust.

	\subsection{Linear optimization with various uncertainty sets}
	Recall from \eqref{2.9} that for the constraints $a_i^Tx\geq b_i; [a_i;b_i]\in \mathcal{U}_i$ in an uncertain RO problem, the uncertainty set $\mathcal{U}_i$ can be written as,
	\begin{align*}
	\mathcal{U}_i=\{[a_i;b_i]=[a_i^0;b_i^0]+\sum_{l=1}^{L_i} \xi_i^LP_i^l\quad:\xi \in \mathcal{Z}_i\}
	\end{align*}
	Obseve that $P_i^l$ can be represented in terms of the partition into entries that correspond to $a_i$ and $b_i$, which we denote $a_i^l$ and $b_i^l$, repectively. So we can write $P_i^l=[a_i^l;b_i^l]$, e.g., $P_1^2=(0,0.5,0)$, where $a_1^2=(0,0.5)$ and $b_1^2=(0)$. 
	Thus 
	\begin{align*}
	\mathcal{U}_i=\{[a_i;b_i]=[a_i^0;b_i^0]+\sum_{l=1}^{L_i} \xi_i^lP_i^l ~:\xi \in \mathcal{Z}_i\}
	\end{align*} is equivalently,
	\begin{align*}
	\mathcal{U}_i=\{[a_i;b_i]=[a_i^0;b_i^0]+\sum_{l=1}^{L_i} \xi_i^l[a_i^l;b_i^l]\quad:\xi \in \mathcal{Z}_i\}
	\end{align*}
	We always assume that the uncertain set is parametrized in affine fashion, by perturbation vector $\xi$ varying in the given perturbation set $\mathcal{Z}$.
	\begin{align}
	\label{3.3.1}
	\mathcal{U}=	\left\{[a^0;b^0]+\sum_{l=1}^L \xi_l[a^l;b^l]: \xi \in \mathcal{Z}\subset\mathbb{R}^L]\right\} \tag{3.3.1}
	\end{align}		
	Our goal now is to build a representation of expressing equivalently the robust counterpart of an uncertain linear inequality as a finite system of explicit convex constraints, with the ultimate goal to use these representations to make it explicit convex program.
	
	\par Consider a family of uncertainty-affected linear inequality
	\begin{equation}
		\label{3.3.2}
	\{A^Tx\geq b\}_{[A;b]\in \mathcal U} \tag{3.3.2}
	\end{equation}
	with the data varying in the uncertainty set (3).\\The corresponding RC is
	\begin{equation}
	\	\label{3.3.3}
	A^Tx\geq b \qquad \forall \left( [a;b]=[a^0;b^0]+\sum_{l=1}^L \xi _l[a^l;b^l]: \xi \in \mathcal Z  \right) \tag{3.3.3}
	\end{equation}
	
	\subsubsection{Tractable form of RC on Interval Uncertainty Set}
	Consider the case that the uncertainty in the constraint data is in inteval. That is, the uncertainty set $\mathcal Z$ is a box \cite{9}. Then without loss of generality, we can normalize the situation by assuming $\mathcal Z$ as follows,
	\begin{equation*}
	\mathcal Z =\text{Box}_1 \equiv \{\xi\in \mathbb R^L:||\xi||_\infty \leq1\}
	\end{equation*}
	In this case, the corresponding RC reads, \cite{9}
	\begin{align*} 
	&[a^0]^Tx+\sum_{l=1}^L \xi_l[a^l]^Tx\geq b^0+\sum_{l-1}^L \xi_lb^l \qquad \quad \forall (\zeta:||\xi||_\infty \leq 1)\\
	\Longrightarrow	&\sum_{l=1}^L \xi_l([a^l]^Tx-b^l)\geq b^0-[a^0]^Tx\qquad \quad \forall (\xi:||\xi_l||\leq1, \: l=1,2,\dots ,L \\
	\Longrightarrow&\min_{-1\leq \xi_l\leq 1} \left[\sum_{l=1}^L \xi_l([a^l]^Tx-b^l)\right]\geq b^0-[a^0]^Tx
	\end{align*}
	Applying KKT method, the concluding maximum is.
	\begin{center}
		$\sum_{l=1}^L |[a^l]^Tx-b^l|$
	\end{center}
	The tractable form of RC \eqref{3.3.3} are represented by the explicit convex constraints,
	\begin{equation*}
	[a^0]^Tx+\sum_{l=1}^L |[a^l]^Tx-b^l|\geq b^0,
	\end{equation*} 
	Considering $|[a^l]^Tx-b^l |\leq u_l$ , $u_l$ the RC admits a representation by a system of linear inequalities
	\[ \begin{cases} 
	-u_l\leq [a^l]^Tx-b^l \geq u_l,\: l=1,2,\dots,L, \\
	[a^0]^Tx+\sum_{l=1}^L u_l\geq b^0
	\end{cases}
	\]
	Now solution of original RO problems is more easier.\\
	
	\textbf{Note:} In the above case, initially $\mathcal{U}$ being infinite the number of constraints become infinite and we call it a semi-definite linear programe. This is generally an intractable class of problems. To solve this type of problems we need to consider the robust counterpart.\\
	The characteristic of an uncertain set $\mathcal{U}$ in an uncertain optimization problem depends on the shape of perturbation set $\mathcal{Z}$ associated to $\mathcal{U}.$\\
	
	\subsubsection{Tractable form of RC on ellipsoidal uncertainty Set}
	
	Consider the uncertainty where $\mathcal Z$ in \eqref{3.3.3} is an ellipse. Here w.l.o.g we can normalize the situation by assuming that $\mathcal Z$ is merely the ball of radius $\Omega$ centered at the origin: \cite{9}
	\begin{equation*}
	\mathcal Z= \text{Ball}_{\Omega}=\{\xi\in\mathbb R^L: ||\xi||_2\leq \Omega\}
	\end{equation*}
	In this case, (3) reads
	\begin{align*} 
	&[a^0]^Tx+\sum_{l=1}^L \xi_l[a^l]^Tx\geq b^0+\sum_{l-1}^L \xi_lb^l \qquad \quad \forall (\xi:||\xi||_2 \leq \Omega) \\ 
	\Longrightarrow&\left[\sum_{l=1}^L \xi_l([a^l]^Tx-b^l)\right]\geq b^0-[a^0]^Tx \qquad \quad \forall (\xi:||\xi||_2 \leq \Omega)\\
	\Longrightarrow&\min_{||\xi||_2\leq \Omega}\left[\sum_{l=1}^L \xi_l([a^l]^Tx-b^l)\right]\geq b^0-[a^0]^Tx\\
	\text{applying KKT method,} \\
	\Longrightarrow&\Omega \sqrt{\sum_{l=1}^L ([a^l]^Tx-b^l)^2}\geq b^0-[a^0]^Tx
	\end{align*}
	The RC admits a representation by the explicit convex constraint.
	\begin{equation*}
	[a^0]^Tx+\Omega \sqrt{\sum_{l=1}^L ([a^l]^Tx-b^l)^2}\geq b^0
	\end{equation*}
	The above inequality is called \textbf{conic quadratic} inequality.
	
	Thus given an uncertain optimization problem with constaints $\{a^Tx\leq b\}_{[a;b]\in \mathcal{U}}$ over an uncertainty-affected set $\mathcal{U}$,  we can find the tractable representation of its RC having explicit convex constraint.\\

	\section{Conclusion}
	In classical linear optimization problems, all the input data are assumed to be known.  However, in real world problems the data is not always certain. Real-world optimization problems that come from design of medical, physical or engineering systems, often contain parameters whose values can not be measured exactly because of various technical difficulties, or due to some incomplete data. Robust optimization is an important tool in optimization that deals with uncertainty.  In sensitive analysis and some other real-world applications, such parameter uncertainties could negatively affect the quality of solutions. We should keep in mind that the proposed robust formulations are theoretically valid only in a neighbourhood of the nominal value. Therefore, their degrees of success will likely be dependent on the quality of parameter estimations and on the magnitude of parameter variations.
	 
	In this paper, we have given a concise introduction and some basic preliminaries of RO that have appeared in the literature to address the concept of robust optimization. We have taken special care to formalize the robust counterparts and the constructions of uncertainty sets and provide specific examples where they have been needed. Numerical experiment is conducted by proposing optimization methods considering some types of uncertain linear optimization problems where uncertainty appears only in the constaints. For solution purpose, we propose our method to solve the numerical RO problems and discussed the comparative results in each case. We restricted the solution methods over interval and ellipsoidal uncertainty sets only to make the discussion easier. To summarize, much opportunities exist for growth and novel research in the field of linear robust optimization, driven by numerical example and practical applications. This paper should serve as a guide to those entering in the exciting and challenging subject of optimization under uncertainties.\\
	
	\textbf{Data Availability} The data that support the findings of this study are generated by the authors and are available from the corresponding author upon request. Also, the author confirms that the data generated in this manuscript are not included in any published paper or in any manuscript that is under consideration for review.\\
	\textbf{Conflict of interest} The authors have no relevant dfinancia or non financial interest to disclose. The authors declair that they have no conflict of interests.

\end{document}